\documentclass[12pt]{amsart}
\usepackage{enumerate}
\usepackage{amsmath, amssymb, amsthm, amsfonts}
\usepackage{tikz}
\usepackage{geometry}
\usepackage[pdftex, bookmarks=true]{hyperref}
\usepackage{bbm} 
\usepackage{url}
\usepackage{cmap} 
\usepackage{verbatim}  
\usepackage{bm}  

\usepackage{marginnote}

\newtheorem{theorem}{Theorem}[section]
\newtheorem{proposition}[theorem]{Proposition}
\newtheorem{corollary}[theorem]{Corollary}
\newtheorem{lemma}[theorem]{Lemma}

\newtheorem*{claim*}{Claim}

\theoremstyle{definition}
\newtheorem{remark}[theorem]{Remark}
\newtheorem*{remark*}{Remark}

\newcommand\ds{\displaystyle}

\newcommand\Pb{\mathbb{P}}  

\newcommand\Nb{\mathbb{N}}
\newcommand\Gb{\mathbb{G}}

\newcommand\Gc{\mathcal{G}}

\newcommand\Oc{\mathcal{O}}  

\newcommand\Zc{\mathcal{Z}}

\newcommand\yt{\widetilde{y}}

\newcommand\dhat{\hat{d}}
\newcommand\comp{\mathrm{c}}

\newcommand\al{\alpha}
\newcommand\be{\beta}
\newcommand\ga{\gamma}
\newcommand\Ga{\Gamma}
\newcommand\de{\delta}
\newcommand\eps{\varepsilon}

\newcommand\sm{\setminus}

\newcommand{\defeq}{\mathrel{\vcenter{\baselineskip0.5ex \lineskiplimit0pt
                     \hbox{\scriptsize.}\hbox{\scriptsize.}}}%
                     =}

\DeclareFontFamily{U}{matha}{\hyphenchar\font45}
\DeclareFontShape{U}{matha}{m}{n}{
  <-6> matha5 <6-7> matha6 <7-8> matha7
  <8-9> matha8 <9-10> matha9
  <10-12> matha10 <12-> matha12
  }{}

\DeclareSymbolFont{matha}{U}{matha}{m}{n}
\DeclareMathSymbol{\Lt}{3}{matha}{"CE}

\title{Star decompositions via orientations}

\author{Viktor Harangi}
\address{HUN-REN Alfr\'ed R\'enyi Institute of Mathematics, Budapest, Hungary}
\email{harangi@renyi.hu}
\thanks{The author was supported by the MTA-R\'enyi Counting in Sparse Graphs ``Momentum'' Research Group, by NRDI 
(grant KKP 138270), and by the Hungarian Academy of Sciences (J\'anos Bolyai Scholarship).}

\begin{document}

\begin{abstract}
A $k$-star decomposition of a graph is a partition of its edges into $k$-stars (i.e., $k$ edges with a common vertex). The paper studies the following problem: given $k \leq d/2$, does the random $d$-regular graph have a $k$-star decomposition (asymptotically almost surely, provided that the number of edges is divisible by $k$)? Delcourt, Greenhill, Isaev, Lidick\'y, and Postle proved the a.a.s.~existence for every odd $k$ using earlier results regarding orientations satisfying certain degree conditions modulo $k$. 

In this paper we give a direct, self-contained proof that works for every $d$ and every $k<d/2-1$. In fact, we prove stronger results. Let $s\geq 1$ denote the integer part of $d/(2k)$. We show that the random $d$-regular graph a.a.s.~has a $k$-star decomposition such that the number of stars centered at each vertex is either $s$ or $s+1$. Moreover, if $k < d/3$ or $k \leq d/2 - 2.6 \log d$, we can even prescribe the set of vertices with $s$ stars, as long as it is of the appropriate size. 
\end{abstract}


\maketitle

\section{Introduction} \label{sec:intro}

For a positive integer $d \geq 3$, let $\Gc_{N,d}$ denote the $N$-vertex random $d$-regular graph, that is, a uniform random graph among all simple $d$-regular graphs on the vertex set $\{1,\ldots, N\}$. We say that $\Gc_{N,d}$ \emph{asymptotically almost surely} (a.a.s.~in short) has a property if the probability that $\Gc_{N,d}$ has this property converges to $1$ as $N \to \infty$. 

Given an integer $k \geq 2$, it is natural to ask whether the edges of $\Gc_{N,d}$ can be partitioned into edge-disjoint stars, each containing $k$ edges. Here we need to restrict ourselves to those $N$ for which the number of edges ($Nd/2$) is divisible by $k$. If such a partition exists with probability $1-o_N(1)$, then we say that $\Gc_{N,d}$ a.a.s.~has a $k$-star decomposition.

This problem behaves very differently in the regimes $k \leq d/2$ and $k>d/2$. In the former case the answer is expected to be positive for every $d,k$, while the latter case is closely related to the well-studied and notoriously difficult problem of accurately determining the independence ratio of random regular graphs.

The study of this problem was initiated in \cite{delcourt2018random}, where the case $d=4, k=3$ was answered affirmatively using second moment calculations. That result was extended recently in \cite{delcourt2023decomposing} where it was shown that the answer is positive whenever $\frac{d}{2}<k<\frac{d}{2}+\frac16\log d$.

In \cite{harangi2025star1} the current author considered the regime $k>d/2$ and proved a.a.s.~existence for $\frac{d}{2}<k<\frac{d}{2}+\big( 1+o_d(1) \big)\log d$, which is asymptotically sharp as $d \to \infty$.

This paper focuses on the regime $2 \leq k \leq d/2$. As we mentioned, the answer is expected to be positive for all $d,k$ in this case. In fact, \cite[Theorem 1.1]{delcourt2018random} claimed---incorrectly---that this follows from a result in \cite{lovaszml2013mod_orientation}. This was later clarified in \cite{delcourt2023decomposing}, where the case of odd $k \leq d/2$ was rigorously deduced from a result of \cite{lovaszml2013mod_orientation} regarding modulo $k$-orientations. 

Note that the problem is very simple when $2k$ divides $d$: any $d$-regular graph has a $k$-star decomposition in this case. Indeed, as it was pointed out in \cite{delcourt2023decomposing}, one can take an Eulerian cycle of $G$ (i.e., a closed walk using every edge exactly once) and direct the edges in the same direction along this walk. The resulting orientation has the property that each in-degree and each out-degree is equal to $d/2$, which is a multiple of $k$ in this case, so one can partition the outgoing edges of each vertex into $k$-stars. We obtain a star decomposition with the property that there are exactly $\sigma \defeq \frac{d}{2k}$ stars centered at each vertex.

What if $\sigma$ is not an integer? The next best thing would be a decomposition with $s$ or $s+1$ stars centered at the vertices, where $s \defeq \lfloor \sigma \rfloor$ is the integer part of $d/(2k)$. In this paper we prove the a.a.s.~existence of such decompositions. 
\begin{theorem} \label{thm:main}
Suppose that $2 \leq k < d/2-1$. Let 
\begin{equation} \label{eq:s&al}
s \defeq \left\lfloor \frac{d}{2k} \right\rfloor ; \quad 
\be \defeq \left\{ \frac{d}{2k} \right\} = \frac{d}{2k} - s .
\end{equation}
Then $\Gc_{N,d}$ a.a.s.~has a $k$-star decomposition with $s$ or $s+1$ stars centered at the vertices. In fact, for any given set $A \subset \{1,\ldots,N\}$ with $|A|=\be N$, it has probability $1-o_N(1)$ that $\Gc_{N,d}$ has a $k$-star decomposition with exactly $s$ stars centered at every $v \notin A$ and exactly $s+1$ stars centered at every $v \in A$.

Moreover, if $k<d/3$ or $k<d/2-2.6 \log d$, then we can arbitrarily prescribe which vertices should have $s$ stars: it holds a.a.s.~for $\Gc_{N,d}$ that the graph has a $k$-star decomposition with the above property for \textbf{every} $A$ with $|A|=\be N$. 
\end{theorem}

As we mentioned, the case $k=d/2$ is easy so for each $d$ there is only one missing case in the regime $k \leq d/2$, namely when $k=\left\lfloor\frac{d-1}{2}\right\rfloor$. 

\subsection*{Proof method}
Using a general result regarding orientations with degree bounds at the vertices, we give a necessary and sufficient condition for a graph $G$ to have a $k$-star decomposition with a prescribed number of stars at each vertex (see Lemma~\ref{lem:gen_sd}). The condition involves bounds on the edge counts of induced subgraphs of $G$. Then we use the first-moment method/counting arguments in the configuration model to prove that random regular graphs satisfy these conditions with high probability.

\subsection*{Notations}
As usual, $V(G)$ denotes the vertex set of a graph $G$, and we write $\deg_G(v)$ or simply $\deg(v)$ for the degree of a vertex $v$, while $G[U]$ stands for the induced subgraph on $U \subseteq V(G)$. By \emph{density} we always refer to the relative size $|U| \big/ |V(G)|$ of a subset $U$. Furthermore, when $G$ is clear from the context, we use the following shorthand notations.
\begin{itemize}
\item Complement: $U^\comp \defeq V(G) \sm U$.
\item Edge count: $e(G)$ denotes the total number of edges, while $e[U] \defeq e\big( G[U] \big)$ is the number of edges inside $U$. Also, for disjoint subsets $U,U' \subseteq V(G)$ we write $e[U,U']$ for the number of edges between $U$ and $U'$. Finally, let $e[v,U] \defeq e[\{v\},U]$.
\end{itemize}
Throughout the paper, the function $h$ stands for $h(x)=-x \log x$, and we write $H(x)$ for $h(x)+h(1-x)$. Also, we use $\sqcup$ for the union of disjoint sets.

\subsection*{Organization of the paper}
In Section~\ref{sec:conditions} we give a general condition that guarantees the existence of a star decomposition in a deterministic (non-random) graph. Section~\ref{sec:subgraphs} contains results about the number of edges of induced subgraphs of random regular graphs. Section~\ref{sec:sd_in_rrg} combines the results of the two previous sections to prove that random regular graphs a.a.s.~have star decompositions. To make the paper as reader-friendly as possible, we moved all the technical computations to Section~\ref{sec:technical}.

\section{Conditions via orientations} \label{sec:conditions}

We start with a general lemma regarding star decompositions in (deterministic) regular graphs.
\begin{lemma} \label{lem:gen_sd}
Let $G$ be a $d$-regular graph with $N$ vertices, and let $k \geq 2$. Suppose that we have a fixed partition of the vertex set:
\[ V(G)= \bigsqcup_{j \geq 0} A_j 
\quad \text{in such a way that } \quad 
\sum_{j \geq 0} j |A_j| = \frac{Nd}{2k} . \]
For a subset $U \subseteq V(G)$ we define 
\[ U_j \defeq U \cap A_j 
\quad \text{and} \quad U'_j \defeq U^\comp \cap A_j .\]
Then the following are equivalent:
\begin{enumerate}[(i)]
\item $G$ has a $k$-star decomposition with exactly $j$ stars centered at every $v \in A_j$ for each $j$;
\item $G$ has an orientation such that for each $j$ and for any $v \in A_j$ the out-degree of $v$ is $jk$;
\item for any $U \subseteq V(G)$ we have 
\begin{equation} \label{eq:cond_U}
e[U] \leq \sum_{j \geq 0} jk |U_j| .
\end{equation}
Furthermore, for any given $U$, \eqref{eq:cond_U} is equivalent to 
\begin{equation} \label{eq:cond_Uc}
e[U^\comp] \leq  \sum_{j \geq 0} (d-jk) |U'_j| 
= d |U^\comp| - \sum_{j \geq 0} jk |U'_j|.
\end{equation}
\end{enumerate}
\end{lemma}
\begin{remark}
Consider the following variant of (ii):
\begin{enumerate}
\item[(ii')] $G$ has an orientation such that for each $j$ and for any $v \in A_j$ the out-degree of $v$ is \textbf{at most} $jk$;
\end{enumerate}
If we only assume $\sum_j j |A_j| \geq Nd/(2k)$, then we still have 
\[ \text{(ii')} 
\Leftrightarrow \text{(iii)} .\] 
Also, \eqref{eq:cond_Uc} still implies \eqref{eq:cond_U} for any given $U$ (although they are not equivalent any more).
\end{remark}
\begin{proof}
(i) $\Leftrightarrow$ (ii) is trivial, while (ii') $\Leftrightarrow$ (iii) is an immediate consequence of \cite[Theorem 1]{frank1976orien}, where orientations with general degree bounds were studied.\footnote{Here we do not rely on anything particularly deep. One could quickly deduce the cited result about orientations from (the integral version of) the max flow min cut theorem.} Furthermore, we clearly have (ii) $\Leftrightarrow$ (ii') under the assumption $\sum_j j |A_j| = Nd/(2k)$.

As for the connection between \eqref{eq:cond_U} and \eqref{eq:cond_Uc} for a given $U$, note that in any $d$-regular graph we have 
\[ 2e[U] + e[U,U^\comp] = d |U| 
\quad \text{and} \quad 
2e[U^\comp] + e[U,U^\comp] = d |U^\comp| .\]
The difference of these equations gives 
\[ e[U^\comp]-e[U] = \frac{d}{2}|U^\comp|-\frac{d}{2}|U| .\]
Using this, as well as $|U_j|+|U'_j|=|A_j|$ and $|U|+|U^\comp|=N$, we get 
\begin{multline*}
e[U^\comp] 
- \left( d |U^\comp| - \sum_{j \geq 0} jk |U'_j| \right) 
- \left( e[U] - \sum_{j \geq 0} jk |U_j| \right) \\
= \frac{d}{2}|U^\comp|-\frac{d}{2}|U| - d|U^\comp| 
+ \sum_{j \geq 0} jk |A_j| 
= k \sum_{j \geq 0} j |A_j| - \frac{Nd}{2} .
\end{multline*}
It follows that \eqref{eq:cond_Uc} $\Rightarrow$ \eqref{eq:cond_U} under $\sum_j j |A_j| \geq Nd/(2k)$, and \eqref{eq:cond_Uc} $\Leftrightarrow$ \eqref{eq:cond_U} under the stronger assumption $\sum_j j |A_j| = Nd/(2k)$. The proof of the lemma and the remark is complete.
\end{proof}
%

Given this result, the following strategy naturally arises for proving the existence of star decompositions in random regular graphs. Suppose that $\al_j \geq 0$ are rational numbers such that 
\begin{equation} \label{eq:al_j_cond}
\sum_{j\geq 0} \al_j = 1 
\quad \text{and} \quad 
\sum_{j\geq 0} j\al_j = \frac{d}{2k} ,
\end{equation}
and take $N \in \Nb$ such that $a_j \defeq N\al_j$ are integers. Furthermore, fix a partition 
\[ \{1,\ldots,N\} = \bigsqcup_{j \geq 0} A_j 
\quad \text{with} \quad 
|A_j|=a_j = N \al_j .\]
Then one can compute the probability that the random graph $\Gc_{N,d}$ satisfies \eqref{eq:cond_U} for some $U \subseteq \{1,\ldots,N\}$ in terms of the sizes $|U'_j|=|U \cap A_j|$. The question is whether we can choose $\al_j$ in such a way that these computations would ensure that \eqref{eq:cond_U} is satisfied for all $U$ with high probability?

Note that if $2k$ divides $d$, then we can simply set 
\[ \al_j = 
\begin{cases}
1 & j=\frac{d}{2k} ;\\
0 & \text{otherwise.}
\end{cases} \]
Then condition \eqref{eq:cond_U} is simply $e[U] \leq \frac{d}{2}|U|$, which is clearly true for all $U$ in any $d$-regular graph. We conclude that every $d$-regular graph has a star decomposition with exactly $d/(2k)$ stars centered at each vertex. This fact was already pointed out in \cite{delcourt2023decomposing} (see the introduction).

From this point on we will assume that $2k$ does not divide $d$ and we set 
\[ s \defeq \left\lfloor \frac{d}{2k} \right\rfloor 
\quad \text{and} \quad r \defeq d-2sk .\]
(Note that $1 \leq r \leq 2k-1$ and $r$ has the same parity as $d$.) This time we need that $\al_j\neq 0$ for at least two indices $j$, otherwise \eqref{eq:al_j_cond} cannot hold. If only $\al_s$ and $\al_{s+1}$ are nonzero, then \eqref{eq:al_j_cond} yields the following values:
\[ \al_j = 
\begin{cases}
\frac{2k-r}{2k} & j=s ;\\
\frac{r}{2k} & j =s+1 ;\\
0 & \text{otherwise.}
\end{cases} \]
Theorem~\ref{thm:main} claims the existence of star decompositions for this specific choice.

\section{Subgraphs with prescribed edge-densities} \label{sec:subgraphs}

In this paper we write $\Gc_{N,d}$ for a random $d$-regular graph on $N$ vertices, that is, a uniform random graph among all $d$-regular simple graphs on the vertex set $\{1,\ldots,N\}$. 

There is a closely related random graph model, obtained via the so-called \emph{configuration model}, that we denote by $\Gb_{N,d}$. Given $N$ vertices, each with $d$ ``half-edges'', the configuration model picks a random pairing of these $Nd$ half-edges, producing $Nd/2$ edges. The resulting random graph $\Gb_{N,d}$ is $d$-regular but it may have loops and multiple edges. A well-known fact is that if $\Gb_{N,d}$ is conditioned to be simple, then we get back $\Gc_{N,d}$. Moreover, for any $d$, the probability that $\Gb_{N,d}$ is simple converges to a positive $p_d$ as $N \to \infty$. It follows that if $\Gb_{N,d}$ a.a.s.~has a certain property, then so does $\Gc_{N,d}$. Therefore, it suffices to prove our a.a.s.~results for $\Gb_{N,d}$.

With the configuration model one can often easily compute or bound the probability of $\Gb_{N,d}$ having various properties. Such probabilities often decay exponentially in the number of vertices $N$, and the rate can be expressed by some kind of entropy.

Note that $\log$ means natural logarithm throughout the paper. As usual, let 
\[ h(x) \defeq 
\begin{cases} 
-x \log x  & \text{if } x\in (0,1]; \\
0 & \text{if } x=0 .
\end{cases} \]
We will also use the shorthand notation: 
\[ H(x) \defeq h(x)+h(1-x) \quad 
\text{for } x \in [0,1] .\]
%
%
%

In order to check condition \eqref{eq:cond_U} of Lemma~\ref{lem:gen_sd} in random graphs, we will need a result regarding the number of edges of induced subgraphs of $\Gb_{N,d}$. Let $x,t \in (0,1)$ be fixed. Given a subset $U \subseteq \{1,\ldots,N\}$ of density $x$ (i.e., $|U|/N=x$), the probability that the average degree of the induced subgraph $\Gb_{N,d}[U]$ is $td$ decays exponentially as $N \to \infty$. We will see that the exponential rate of decay  is $d \cdot F(x,t)$, where  
\begin{equation} \label{eq:F_formula}
F(x,t) \defeq 
\frac12 h\big( tx \big) 
+ h\big( (1-t)x \big) 
+ \frac12 h\big( 1-(2-t)x \big) 
- H(x). 
\end{equation}
The precise result is the following.
\begin{lemma} \label{lem:av_deg_gen}
Let $x,t \in (0,1)$ such that $(2-t)x\leq 1$. Suppose that $U$ is a fixed subset of $\{1,\ldots,N\}$ of size $xN$. Then the probability that the induced subgraph $\Gb_{N,d}[U]$ has $xtNd/2$ edges (i.e., its average degree is $td$) is at most 
\begin{equation} \label{eq:prob_formula}
(Nd)^{\Oc(1)} \exp\bigg( Nd \cdot F(x,t) \bigg) .
\end{equation}
\end{lemma}
\begin{proof}
Fix $d$ and $N$. Let $M$ be a positive integer and $r \geq 0$ such that $rM/2$ is an integer and $M(2d-r) \leq Nd$. By $P_{M,r}$ we denote the probability that for a given subset $U \subseteq \{1,\ldots,N\}$ of size $M$, the random graph $\Gb_{N,d}$ has $rM/2$ edges inside $U$. It is easy to see that 
\[ P_{M,r} = \binom{Md}{Mr} 
\binom{(N-M)d}{M(d-r)} \,  
\frac{(Mr)!}{(Mr)!!} \, 
\frac{(Nd-M(2d-r))!}{(Nd-M(2d-r))!!} \, 
\big( M(d-r) \big)! \, 
\frac{(Nd)!!}{(Nd)!} ,\]
which can be rewritten with binomial and multinomial coefficients as 
\begin{equation} \label{eq:P_nomial_version}
P_{M,r} = 
\underbrace{\frac{(M(d-r))!}{\big( (M(d-r))!! \big)^2}}_{
<1 } \;
\binom{\frac{Nd}{2}}{\frac{Mr}{2} \;\; \frac{M(d-r)}{2} \;\; \frac{M(d-r)}{2} \;\;  \frac{Nd-M(2d-r)}{2} } 
\Bigg/ \binom{Nd}{Md} .
\end{equation}
It is well known\footnote{One may prove this by considering the sum $n^n=(k_1+\cdots+k_s)^n=\sum \binom{n}{j_1 \; \ldots \; j_s} k_1^{j_1} \cdots k_s^{j_s}$ with $\binom{n+s-1}{s-1}$ terms and noting that we get the largest term for $j_i=k_i$.} that multinomial coefficient can be estimated using entropy:
%
\[ \binom{n+s-1}{s-1}^{-1} \exp\left(n \sum_{i=1}^s h(k_i/n) \right) 
\leq \binom{n}{k_1 \; k_2 \; \ldots \; k_s} 
\leq \exp\left(n \sum_{i=1}^s h(k_i/n) \right) .\]
Setting $x \defeq M/N$ and $t \defeq r/d$, these bounds immediately yield \eqref{eq:prob_formula} from \eqref{eq:P_nomial_version}. 
\end{proof}
We will also need the following result, which can be derived easily from \eqref{eq:P_nomial_version}.
\begin{lemma} \label{lem:av_deg_2}
For every integer $d \geq 3$ and every real number $\dhat > 2$ there exists $\eps>0$ such that $\Gb_{N,d}$ a.a.s.~has no induced subgraph on at most $\eps N$ vertices with average degree at least $\dhat$.
\end{lemma}
\begin{proof}
Consider the random graph $\Gb_{N,d}$ produced by the configuration model. Given a positive integer $M \leq N$ and a rational number $r \in [0,d]$ with $rM/2 \in \Nb$, we define $\Zc_{M,r}$ to be the expected number of sets $U \subset \{1,\ldots,N\}$ with $|U|=M$ and $e[U]=r M/2$. So, with our previous notation $P_{M,r}$, we have  
\[ \Zc_{M,r} = \binom{N}{M} P_{M,r} .\]
Using that $n! \geq (n/e)^n$ and $n(n-1) \cdots (n-m+1) \geq (n/e)^m$ and trivial estimates we get from \eqref{eq:P_nomial_version} that 
\begin{equation} \label{eq:Z_bound}
\Zc_{M,r} <
\frac{N^M}{\left(\frac{M}{e}\right)^M} \;
\frac{\left(\frac{Nd}{2}\right)^{M(d-r/2)} }{\left( \frac{Mr}{2e} \right)^{Mr/2} \left( \frac{M(d-r)}{2e} \right)^{M(d-r)} } \;
\frac{(Md)^{Md}}{\left(\frac{Nd}{e}\right)^{Md}} 
< \left( e^{2d} d^d \left( \frac{M}{Ne} \right)^{r/2-1} \right)^M .
\end{equation}
We need to consider the case $r \geq \dhat$. Note that $\dhat/2-1$ is positive since $\dhat>2$. Therefore, we can choose $\eps>0$ in such a way that 
\[ e^{2d} d^d \left( \frac{\eps}{e} \right)^{\dhat/2-1} 
< \frac{1}{2} .\]
For a given $N$, let us consider all positive $M \leq \eps N$, and then, for a given $M$, all rational $r$ satisfying $\dhat \leq r \leq d$ and $rM/2 \in \Nb$. For any such $M$ and $r$, \eqref{eq:Z_bound} implies that $\Zc_{M,r} < 2^{-M}$. Note that there are at most $Md/2$ possible values of $r$ for a given $M$, and hence 
\[ \Zc_{M} \defeq \sum_r \Zc_{M,r} < \frac{Md}{2} 2^{-M} .\] 
For a given $\delta>0$, choose $M_0$ in such a way that 
\[ \sum_{M: \, M_0 < M \leq \eps N} \Zc_{M} < 
\sum_{M > M_0} \frac{Md}{2} 2^{-M} < \de/2.\]
As for $M \leq M_0$, there are finitely many terms $\Zc_{M,r}$ in this range, each converging to $0$ as $N \to \infty$, so 
\[ \sum_{M \leq M_0} \Zc_{M} < \de/2 \quad 
\text{for sufficiently large $N$.} \]
It follows that $\sum_{M \leq \eps N} \Zc_M <\de/2 + \de/2 = \de$. By Markov's inequality we get that $\Gb_{N,d}$ fails to have the claimed property with probability less than $\de$.
\end{proof}

Next we list some properties of the function $F(x,t)$. The proofs, which consist of mostly straightforward calculations, are postponed until Section~\ref{sec:F_properties}. 
\begin{proposition} \label{prop:F_properties}
The function $F(x,t)$, defined in \eqref{eq:F_formula} for any $x,t \in [0,1]$ with $(2-t)x \leq 1$, is clearly continuous. It also has the following properties.
\begin{enumerate}[(i)]
\item $F(x,t) \leq 0$ with equality if and only if $x=0$ or $x=t$.
\item For any fixed $x \in (0,1)$, the function $t \mapsto F(x,t)$ is strictly monotone decreasing (and concave) on $[x,1]$ with $F(x,x)=0$ and $F(x,1)=-H(x)/2$. It has the following derivative:
\[ \partial_t F(x,t) 
= \frac12 x \log \left( 1- \frac{t-x}{t \big( 1-(2-t)x \big)} \right) . \]
\item For any fixed $t \in (0,1)$, the function $x \mapsto F(x,t) \big/ H(x)$ maps $(0,t]$ onto $(-t/2,0]$ strictly monotone increasingly and continously. 
\item Symmetry: for $x'=1-x$ and $t'\in (0,1)$ such that $x(1-t)=x'(1-t')$ we have 
\[ F(x,t)=F(x',t') 
\quad \text{and} \quad 
H(x)=H(x') .\]
\end{enumerate}
\end{proposition}

We combine the results of this section to prove the following.
\begin{corollary} \label{cor:all_subgraphs}
Let
\begin{equation} \label{eq:Fd_def}
F_d(x,t) \defeq d \cdot F(x,t) + H(x) ,
\end{equation}
and suppose that for some $0<x_0<t_0<1$ we have
\begin{equation} \label{eq:cor_cond}
F_d(x_0,t_0) < 0 
\text{, that is, }
F(x_0,t_0) \big/ H(x_0) < -\frac{1}{d}
\end{equation}
Then it holds a.a.s.~for $\Gb_{N,d}$ that all induced subgraphs on at most $x_0 N$ vertices have average degree at most $t_0 d$.
\end{corollary}
\begin{proof}
If $xN \in \Nb$, then the number of subsets $U\subset \{1,\ldots,N\}$ of size $xN$ is
\[ \binom{N}{Nx} \leq \exp\bigg( N \cdot H(x) \bigg) .\]
So, by Lemma~\ref{lem:av_deg_gen}, the expected number of induced subgraphs with size $xN$ and average degree $td$ is at most
\[ \binom{N}{Nx} (Nd)^{\Oc(1)} \exp\bigg( Nd \cdot F(x,t) \bigg) = (Nd)^{\Oc(1)} \exp\bigg( Nd \cdot F(x,t) + N \cdot H(x) \bigg) .\]
%
%
According to Proposition~\ref{prop:F_properties}(iii), condition \eqref{eq:cor_cond} implies $t_0/2>1/d$. Thus $t_0 d>2$, and hence, by Lemma~\ref{lem:av_deg_2}, there exists a positive $\eps>0$ such that a.a.s.~$\Gb_{N,d}$ has no induced subgraph on at most $\eps N$ vertices with average degree above $t_0 d$.

By Proposition~\ref{prop:F_properties}(ii-iii), we know that $\frac{F(x,t)}{H(x)}$ is monotone increasing in $x$ and monotone decreasing in $t$. Therefore, 
\begin{multline*}
\eqref{eq:cor_cond} \Rightarrow 
\frac{F(x_0,t_0)}{H(x_0)} < - \frac{1+\de}{d}  
\;\text{(for some $\de>0$)} 
\Rightarrow
\frac{F(x,t)}{H(x)} < - \frac{1+\de}{d} \; (\forall x \leq x_0; \forall t \geq t_0) \\
\Rightarrow
d \cdot F(x,t) + H(x) \leq - \de \, H(x)  \quad (\forall x \leq x_0; \forall t \geq t_0) \\
\Rightarrow
d \cdot F(x,t) + H(x) \leq - \underbrace{\de \min\big( H(\eps), H(x_0) \big)}_{\de' \defeq}  \quad (\forall \eps \leq x \leq x_0; \forall t \geq t_0) .
\end{multline*}
Since there are $(Nd)^{\Oc(1)}$ ways to choose $x$ and $t$, we get that the expected number of sets $U \subset \{1,\ldots,N\}$ of size between $\eps N$ and $x_0 N$ with average degree at least $t_0 d$ is at most 
\[ (Nd)^{\Oc(1)} \exp\big( -N \de' \big) \to 0 
\quad \text{as } N \to \infty .\]
By Markov's inequality, the proof is complete.
\end{proof}

\section{Star decompositions in random regular graphs}
\label{sec:sd_in_rrg}

In this section we deduce Theorem~\ref{thm:main} from the results of the previous sections. We start with the regime where we can even prescribe which vertices should have $s$ stars.
\begin{theorem} \label{thm:strong_cond}
Let $d=2sk+r$ with $r<2k$. If 
\begin{equation} \label{eq:strong_cond}
F_d( x_0,t_0 ) < 0 
\quad \text{ for } 
x_0=\frac{r}{2k} \text{ and }
t_0=  \frac{d-2k+r}{d},
\end{equation}
then $\Gb_{N,d}$ a.a.s.~(as $N \to \infty$ with $Nd/(2k) \in \Nb$) has the following property: for every partition $A_s \sqcup A_{s+1}=\{1,\ldots,N\}$ with $|A_{s+1}|=rN/(2k)$ there exists a $k$-star decomposition with exactly $j$ stars centered at the vertices in $A_j$ ($j=s,s+1$).

We will refer to \eqref{eq:strong_cond} as the \emph{strong condition}. Later we will see that it is satisfied\footnote{See the tables in Section~\ref{sec:strong} for our exact thresholds for specific values of $d$.} whenever $k<d/3$ or $k<d/2-2.6 \log d$.
\end{theorem}
\begin{proof}
We want to use Corollary~\ref{cor:all_subgraphs} with two settings:
\[ x_0=\frac{r}{2k}; \quad 
t_0=  \frac{d-2k+r}{d} = 1-\frac{2k-r}{d} \]
and
\[ x'_0=1-x_0=1-\frac{r}{2k}=\frac{2k-r}{2k}; \quad
t'_0= 1-\frac{r}{d} .\]
It is easy to check that $x_0(1-t_0)=x'_0(1-t'_0)=r(2k-r)/(2kd)$ and hence we have $F(x_0,t_0)=F(x'_0,t'_0)$ and $H(x_0)=H(x'_0)$ according to Proposition~\ref{prop:F_properties}(iv). It follows that $F_d(x_0,t_0)=F_d(x'_0,t'_0)$, which is negative according to the strong condition \eqref{eq:strong_cond}. So the corollary can be indeed applied both for $x_0,t_0$ and for $x'_0,t'_0$. 

Note that $t'_0d/2=(d-r)/2=sk$ is the smaller of the two relevant coefficients in \eqref{eq:cond_U}. Similarly, $t_0d/2=(d-2k+r)/2=d-(s+1)k$ is the smaller of the two relevant coefficients in \eqref{eq:cond_Uc}. 

Now let $U \subset \{1,\ldots,N\}$ be an arbitrary subset. We distinguish two cases based on the density $|U|/N$. 

\textbf{First case: } $|U|/N \leq x'_0$. By Corollary~\ref{cor:all_subgraphs} it holds a.a.s.~for $\Gb_{N,d}$ that for any such $U$ the induced subgraph on $U$ has average degree at most $t'_0$, that is, 
\[ e[U] \leq \frac{dt'_0}{2} |U|  = sk |U|.\]
It follows that condition \eqref{eq:cond_U} is satisfied for any such $U$ and for any partition $A_s \sqcup A_{s+1}$ because the other coefficient $(s+1)k$ is larger than $sk$.

\textbf{Second case: } $|U|/N > x'_0 \Leftrightarrow |U^\comp|/N< 1-x'_0=x_0$. Similarly, by Corollary~\ref{cor:all_subgraphs} we have  
\[ e[U^\comp] \leq \frac{dt_0}{2} |U^\comp| 
= \big( d- (s+1)k \big) |U^\comp| ,\]
and hence condition \eqref{eq:cond_Uc} is satisfied for any such $U$ and for any partition $A_s \sqcup A_{s+1}$.

In conclusion, for each $U$ either \eqref{eq:cond_U}, or \eqref{eq:cond_Uc} is satisfied. We know, however, that \eqref{eq:cond_U} and \eqref{eq:cond_Uc} are actually equivalent, and hence we showed that it holds a.a.s.~for $\Gb_{N,d}$ that Lemma~\ref{lem:gen_sd} can be applied for all partitions with appropriate sizes.
\end{proof}

Now we turn to the proof of Theorem~\ref{thm:main}. Here, we present the main points of the arguments. The technical parts are postponed until Section~\ref{sec:technical}. 

Theorem~\ref{thm:strong_cond} settles the case when the strong condition holds. We will see in Section~\ref{sec:strong} that the strong condition \eqref{eq:strong_cond} is satisfied if $k<d/3$ or $k<d/2-2.6 \log d$. From this point on we assume that the strong condition does not hold, and hence $k > d/3$ and $k>d/2-2.6 \log d$. In particular, we have $s=1$ and $r \leq k$. Since $k < d/2-1$, we can choose $r \in \{3,4,\ldots,k\}$ such that $d=2k+r$. In this case the density of vertices with $s=1$ and $s+1=2$ stars is 
\[ \al_1 = 1-\frac{r}{2k} 
\quad \text{and} \quad 
\al_2 = \frac{r}{2k} 
\text{, respectively.} \]
Let $A_1 \sqcup A_2$ be any fixed partition of $\{1,\ldots,N\}$ with $|A_j|=\al_j N$ ($j=1,2$). 

Again, we will use a first moment calculation to show that condition \eqref{eq:cond_U} holds for all $U$ with high probability. This time we need to be more careful and take into account the sizes of the intersections $U_j=U \cap A_j$ (rather than simply considering the smaller coefficients in \eqref{eq:cond_U} and \eqref{eq:cond_Uc} as we did under the strong condition). However, we still have a regime where we do not need to worry about the intersection sizes.

Let 
\[ t_0 \defeq \frac{2r}{d} 
\quad \text{and} \quad
t'_0 \defeq \frac{2k}{d} \]
be the $t$-values corresponding to those smaller coefficients ($d-2k$ and $k$, respectively) in the sense that 
\[ \frac{t_0 d}{2} = r= d-2k 
\quad \text{and} \quad
\frac{t'_0 d}{2} = k .\]

Suppose that $x_-$ and $x_+$ are chosen in a way that 
\begin{equation} \label{eq:x_minus_plus}
F_d(x_-,t_0) < 0 
\quad \text{and} \quad 
F_d(1-x_+,t'_0) < 0 .
\end{equation}
Then the same argument gives that it holds a.a.s. for $\Gb_{N.d}$ that, no matter what partition $A_1 \sqcup A_2$ we have, all subsets $U$ below density $x_-$ satisfy condition \eqref{eq:cond_U} and all subsets $U$ above density $x_+$ satisfy condition \eqref{eq:cond_Uc}. Consequently, we may assume that we have a subset $U$ with density in $(x_-,x_+)$.\footnote{In fact, the strong condition is satisfied precisely when we can choose $x_-=x_+=\al_2$.} 

Given a fixed partition $A_1 \sqcup A_2$, we say that a set $U \subset \{1,\ldots,N\}$ has (intersection) profile $(x_1,x_2)$ if $|U_j|=|U \cap A_j|=x_j N$, $j=1,2$. Note that $x_1 \in [0,\al_1]$ and $x_2 \in [0,\al_2]$. 

First we determine the number of sets with profile $(x_1,x_2)$. We have the following for the exponential rate of the number of subsets of size $xN$ of a set of size $\al N$ ($0 \leq x \leq \al \leq 1)$:
\[ \binom{N\al}{Nx} 
\leq \exp\left( N\al \cdot H\left( \frac{x}{\al} \right) \right)
= \exp\bigg( N \big( h(x)+h(\al-x)-h(\al) \big) \bigg) .\] 
For brevity we will use the notation 
\[ g(\al,x) \defeq h(x)+h(\al-x)-h(\al) .\]
Therefore, the number of sets $U$ with profile $(x_1,x_2)$ is of rate $g(\al_1,x_1)+g(\al_2,x_2)$.

For future reference, let us remark that 
\begin{equation} \label{eq:g_bound}
g(\al,x) \leq h(x) + x \log(e \al) \leq h(x) + x .
\end{equation}

We also need the probability that a set $U$ with profile $(x_1,x_2)$ does not satisfy condition \eqref{eq:cond_U}. Note that \eqref{eq:cond_U} is equivalent to $G[U]$ having average degree at most $td$, where
\[ t= \frac{2(k+r) x_1 + 2r x_2}{d(x_1+x_2)} 
= \frac{2r}{d} + \frac{2k x_1}{d(x_1+x_2)} .\]
%
Note that this holds automatically if $t \geq 1$, or equivalently if $x_1/x_2 \geq \al_1/\al_2$. If $t<1$, then the probability that the average degree of $G[U]$ is above $td$ has exponential rate at most $d \cdot F(x_1+x_2,t)$; see Lemma~\ref{lem:av_deg_gen} and Proposition~\ref{prop:F_properties}(ii). Therefore, the exponential rate of the expected number of sets $U$ that have profile $(x_1,x_2)$ and do not satisfy \eqref{eq:cond_U} is given by the following function: 
\begin{equation} \label{eq:eta}
\eta(x_1,x_2) \defeq d \cdot F\left( x_1+x_2, 
\frac{2r}{d} + \frac{2k x_1}{d(x_1+x_2)} \right)+g(\al_1,x_1)+g(\al_2,x_2) .
\end{equation}
Our goal is to show that this is negative for every possible profile. The precise statement that we will prove in Section~\ref{sec:weak} is the following.
\begin{lemma} \label{lem:profile}
Given any $k \geq 2$ and $r\geq 3$, there exist $0<x_-<x_+<1$ satisfying \eqref{eq:x_minus_plus} such that we have $\eta(x_1,x_2)<0$ for any 
\begin{equation} \label{eq:region}
x_1 \in [0,\al_1] \text{ and }
x_2 \in [0,\al_2] \text{ with }
x_1+x_2 \in [x_-,x_+] \text{ and } 
\frac{x_1}{x_2} \leq \frac{\al_1}{\al_2} .
\end{equation}
\end{lemma}
Here we show why this lemma completes the proof of Theorem~\ref{thm:main}. Since $\eta$ is continuous and the region defined by \eqref{eq:region} is compact, there exists $\de'>0$ such that $\eta(x_1,x_2) < - \de'$ everywhere in this region. For a given $N$, the number of possible profiles is $(Nd)^{\Oc(1)}$, and hence the expected number of sets $U$ of density in $[x_-,x_+]$ for which condition \eqref{eq:cond_U} fails is at most 
\[ (Nd)^{\Oc(1)} \exp\big( -N \de' \big) \to 0 
\quad \text{as } N \to \infty .\]
By Markov’s inequality, the proof of Theorem~\ref{thm:main} is complete.

Note that Lemma~\ref{lem:profile} fails to be true when $r=1,2$, and that is the reason why we need the condition $k<d/2-1$. We will rigorously prove the lemma for $r \geq 3$ in Section~\ref{sec:technical}. The proof strategy is that we will first check that $\eta(x_1,x_2)$ is negative in the special cases $x_1=0$ or $x_2=\al_2$. Then we will show that the maximum of $\eta$ over the region \eqref{eq:region} must be (very) close to one of these special cases. Using this we will conclude that the maximum must be negative as well.

\section{Technical computations} \label{sec:technical}

This section contains the technical parts of the proofs from preceding sections. 
\subsection{Properties of \textit{F}} 
\label{sec:F_properties}

We start with rigorously proving the properties of the function $F(x,t)$ listed in Proposition~\ref{prop:F_properties}. 

To see (i) we notice that the probabilities 
\[ p_{00} \defeq xt; \quad 
p_{01} = p_{10} \defeq x(1-t); \quad 
p_{11} \defeq 1-(2-t)x \]
define a distribution corresponding to 
$p_{ij} = \Pb(X=i \,\&\, Y=j)$; $i,j \in \{0,1\}$, where both marginals are distributed
as 
\[ p_0=x; \quad p_1=1-x .\]
Therefore the Shannon entropy 
\[ H(X,Y) = \sum_{i,j \in \{0,1\}} h\big( p_{ij} \big) = h\big( tx \big) 
+ 2h\big( (1-t)x \big) 
+ h\big( 1-(2-t)x \big) \]
 of the joint distribution of $X,Y$ is at most 
\[ H(X)+H(Y) = 2 H(x) ,\]
with equality if and only if $X$ and $Y$ are independent ($p_{00}=p_0^2$), that is, if $xt=x^2$. It follows that $F(x,t) \leq 0$ with equality if and only if $x=0$ or $x=t$.

For (ii) we need to differentiate $F$ w.r.t.~the second variable $t$:  
\begin{align*}
\partial_t F(x,t) &= \frac12 x \bigg( - \log(t x) + 2 \log( (1-t)x) - \log(1-(2-t)x) \bigg) \\
&= \frac12 x \log \frac{(1-t)^2 x}{t \big( 1-(2-t)x \big)} \\
&= \frac12 x \log \left( 1- \frac{t-x}{t \big( 1-(2-t)x \big)} \right) ,
\end{align*}
which is clearly negative for $t \in (x,1)$, implying monotonicity. We will not need concavity but it would follow easily by considering the second derivative.

As for (iii), first we use the identity $h(ab)=a \cdot h(b) + b \cdot h(a)$ to rewrite $F(x,t)$ as follows:
\begin{multline*}
F(x,t) = \frac12 x h(t) + \frac12 t h(x) 
+ \frac12 h\big( 1-(2-t)x \big) 
+ x h(1-t) + (1-t)h(x) -h(x) - h(1-x) \\
= -\frac12 t h(x) + \frac12 x h(t)
+ \frac12 h\big( 1-(2-t)x \big) + x h(1-t) - h(1-x) .
\end{multline*} 
Note that $x / h(x) \to 0$ as $x \to 0+$ and $h(1-z)\leq z$. So for any fixed $t \in (0,1)$ we have  
\[ \frac{F(x,t)}{h(x)} \to -\frac{t}{2} 
\quad \text{and} \quad 
\frac{H(x)}{h(x)} \to 1 
\quad \text{as } x \to 0+ .\]
We conclude that 
\begin{equation} \label{eq:t_half}
\lim_{x \to 0+} \frac{F(x,t)}{H(x)} = -\frac{t}{2} .
\end{equation} 
Now we fix $t \in (0,1)$ and a constant $c>0$ and consider the function 
\[ G(x) \defeq F(x,t) + c H(x) .\]
Differentiating (w.r.t.~$x$) gives:
\begin{multline*}
G'(x) = 
-\frac12 t \log(t x) - (1-t)\log( (1-t)x) + \frac12 (2-t)\log(1-(2-t)x) \\
- (1- c) \big( -\log(x) + \log(1-x) \big) \\ 
= \frac12 h(t)+h(1-t) + (1-t/2)\log(1-(2-t)x) 
+ \left(t/2-c \right) \log(x) 
- (1- c) \log(1-x) .
\end{multline*} 
Then the second derivative is
\begin{align*}
G''(x) &= 
\frac{-\frac12(2-t)^2}{1-(2-t)x} 
+ \left(\frac{t}{2} -c \right) \frac{1}{x} 
+ (1- c) \frac{1}{1-x} \\
&= \frac{ \left( \frac{t}{2} -c \right) 
- x(2-t) \left(\frac12-c\right) }
{x(1-x)\big( 1-(2-t)x \big)}.
\end{align*} 
Therefore, $G''(x) = 0$ if and only if $x$ is equal to 
\[ \hat{x} \defeq 
\frac{\frac{t}{2}-c}{\left( \frac{1}{2} -c \right) (2-t)} .\]
Therefore, if $0<c<t/2$, then $G(x)$ is convex on $(0,\hat{x})$ and concave on $(\hat{x}_t,t)$. Furthermore, on the one hand, due to \eqref{eq:t_half}, $G(x)$ is negative for small $x$. On the other hand, $G$ is positive at $x=t$ as $G(t)=F(t,t)+c F(t)=c F(t) >0$. It follows that $G$ has a unique root. In other words, $F(x,t)/H(x)=-c$ for exactly one $x \in (0,t)$. If $c> t/2$, then $G$ is concave on the entire $(0,t)$ and positive at both endpoints. In conclusion, $G$ is positive everywhere and hence $F(x,t)/H(x)=-c$ has no solution. 

We conclude that $x \mapsto F(x,t)/H(x)$ must be a bijection between $(0,t)$ and (-t/2,0), and the claim follows from the continuity of $F(x,t)/H(x)$.

Finally, (iv) simply follows from the following equalities:
\begin{align*}
xt &=1-(2-t')x' ;\\
x(1-t) &=x'(1-t') ;\\
1-(2-t)x &= x't'.
\end{align*}
%

\subsection{Estimating \textit{F}} 
\label{sec:F_estimates}

Here we give sharp estimates for $F(x,t)$ when $x$ and $t$ are close to $0$.

First of all, using power series we get 
\[ h(1-z)=(1-z)\log \frac{1}{1-z} 
= (1-z) \sum_{i=1}^{\infty} \frac{z^i}{i} 
= z - \sum_{i=2}^{\infty} \frac{z^i}{i(i-1)} .\]
We can conclude that 
\[ z-\frac{z^2}{2}-\frac{z^3}{4} \leq h(1-z) \leq z-\frac{z^2}{2}-\frac{z^3}{6} .\]
The upper bound holds for any $z \in [0,1]$, while the lower bound holds on the interval $z \in [0,0.6]$.

Simple manipulations show that 
\begin{equation*} 
\frac12 h\big( tx \big) 
+ h\big( (1-t)x \big) 
- h(x) 
= \frac12 xt \log\left(\frac{x}{t}\right) + x h(1-t)
\end{equation*}
Furthermore, we have the following bounds:
\begin{align*}
-h(1-x) & \leq = - x + \frac12 x^2 + \frac14 x^3 
\quad \text{if } x \in [0,0.6] ;\\
x h(1-t) & \leq xt - \frac12 xt^2 - \frac16 xt^3 ;\\
\frac12 h\big( 1-(2-t)x \big) & \leq 
\frac 12 (2-t)x - \frac14 (2-t)^2x^2 - \frac{1}{12}(2-t)^3 x^3 \\ 
&= x-\frac12 xt - x^2 + x^2t - \frac14 x^2t^2 - \frac{1}{12}(2-t)^3 x^3 .
\end{align*}
Putting these together, we get the following.
\begin{proposition} \label{prop:F_estimate}
For $0 \leq x \leq 0.6$ and $x \leq t \leq 1$ we have 
\begin{equation*} 
F(x,t) \leq 
\frac12 xt \left( 1-\frac{x}{t} + \log\frac{x}{t} \right) 
+ x^2t -\frac12xt^2 - \frac16 xt^3 
+ \frac14 \left( 1- \frac{(2-t)^3}{3} \right)x^3 . 
\end{equation*}
\end{proposition}
For small $x,t$ the main term here is $\frac12 xt \,\varphi(x/t)$, where $\varphi(z) \defeq 1-z+\log z$. Note that $\varphi$ is negative and monotone increasing on $(0,1)$.
In the following regime $F(x,t)$ can actually be upper bounded by the main term.
\begin{proposition} \label{prop:F_main_term}
For $0 \leq x \leq 0.2$ and $t \geq 2x/(1+x)$ we have 
\begin{equation*} 
F(x,t) 
\leq \frac12 xt \, \varphi\left( \frac{x}{t} \right)
= \frac12 xt \left( 1-\frac{x}{t} + \log\frac{x}{t} \right) .
\end{equation*}
\end{proposition}
\begin{proof}
Due to Proposition~\ref{prop:F_estimate}, we need to show that the remaining part of the sum is at most $0$, that is:
\begin{equation} \label{eq:extra_terms}
\frac16xt \bigg( 3(2x-t) - t^2 \bigg) 
+ \frac14 \left( 1- \frac{(2-t)^3}{3} \right)x^3 \leq 0. 
\end{equation}
Let 
\[ t_0 \defeq \frac{2x}{1+x} .\]
First we show that \eqref{eq:extra_terms} holds at $t=t_0$. We have 
\[ 2x-t_0= \frac{2x^2}{1+x} 
\text{, and hence } 
3(2x-t_0)-t_0^2 = \frac{6x^2(1+x)-4x^2}{(1+x)^2} 
= \frac{2x^2(1+3x)}{(1+x)^2}
\]
We get that 
\begin{equation} \label{eq:extra_term1}
\frac16xt_0 \bigg( 3(2x-t_0) - t_0^2 \bigg) 
= \frac{2x^4(1+3x)}{3(1+x)^3} 
\leq \frac{2x^4(1+3x)}{3(1+3x)} 
= \frac23 x^4 .
\end{equation}
It is easy to check that   
\[ \frac{1}{(1+x)^3} > x + \frac38 
\text{ for } 0 \leq x \leq 0.2 ,\]
and hence 
\begin{equation} \label{eq:extra_term2}
\frac14 \left( 1- \frac{(2-t_0)^3}{3} \right)x^3 
= \frac14 \left( 1- \frac{8/(1+x)^3}{3} \right)x^3 
\leq \frac14 \left( 1- \frac{8x+3}{3} \right)x^3 
= - \frac23 x^4 .
\end{equation}
Therefore, \eqref{eq:extra_term1} and \eqref{eq:extra_term2} give \eqref{eq:extra_terms} for $t=t_0$. It can be seen easily that the derivative of \eqref{eq:extra_terms} w.r.t.~$t$ is
 \[ x \left( x-2t-\frac12 t^2 + \frac14 x^2(2-t)^2 \right) ,\]
which is negative on $[t_0,1]$ for $x \leq 0.2$. Thus \eqref{eq:extra_terms} indeed holds for every $t \geq t_0$.
\end{proof}

\subsection{The strong condition}
\label{sec:strong}

Now we prove that condition \eqref{eq:strong_cond} of Theorem~\ref{thm:strong_cond} is satisfied if $k<d/3$ or $k<d/2-2.6 \log d$. This condition can be checked easily for specific values of $d$ and $k$: we simply need to evaluate $F(x_0,t_0)/H(x_0)$ at $x_0=r/(2k)$ and $t_0=(d-2k+r)/d$, and check if it is less than $-1/d$. The next tables show the values $k_d^{\textrm{sc}}$ up to which the strong condition holds. 

\medskip

\begin{tabular}{l||r|r|r|r|r|r|r|r|r|r|r|r|r|r|r|r|r} 
$d$                 & 13 & 14 & 15 & 16 & 17 & 18 & 19 & 20 & 21 & 22 & 23 & 24 & 25 & 26 & 27 & 28 & 29 \\
\hline
$k_d^{\textrm{sc}}$ & 3  &  4 &  4 &  4 &  5 &  5 &  5 &  6 &  6 &  7 &  7 &  7 &  8 &  8 &  9 &  9 & 10
\end{tabular}

\medskip

\begin{tabular}{l||r|r|r|r|r|r|r|r|r|r|r|r|r|r|r|r|r} 
$d$                 & 30 & 40 & 50 & 60 & 70 & 80 & 90 & 100 & 110 & 120 & 130 & 140 & 150 & 160 & 500 \\
\hline
$k_d^{\textrm{sc}}$ & 10 & 14 & 19 & 24 & 28 & 33 & 38 &  42 &  47 &  52 &  57 &  62 &  67 &  71 & 239
\end{tabular}

\medskip

One can quickly check with a computer that for $d \leq 500$ the strong condition indeed holds provided that $k<\max( d/3, d/2 - 2.6 \log d)$. Therefore, in what follows we may assume that $d>500$ whenever it is needed.

Recall that $d=2sk+r$ and 
\[ \be 
= \left\{ \frac{d}{2k} \right\} 
= \frac{r}{2k} .\]

We start with the case $k<d/4$. Then $s \geq 2$, and hence 
\[ \frac{d-2k+r}{d}
= 1-\frac{2k-r}{2sk+r} 
= 1-\frac{2k(1-\be)}{2k(s+\be)} 
= 1-\frac{1-\be}{s+\be} 
\geq 1-\frac{1-\be}{2+\be} 
= \frac{1+2\be}{2+\be} .\]
According to Proposition~\ref{prop:F_properties}(ii) $F(x,t)$ is monotone decreasing in $t$, and hence 
\[ F\left( \frac{r}{2k}, \frac{d-2k+r}{d} \right) 
\leq F\left( \be, \frac{1+2\be}{2+\be} \right) .\]
Therefore, it suffices to prove that  
\[ F\left( \be, \frac{1+2\be}{2+\be} \right) \bigg/ H(\be)
< -\frac{1}{d} .\]
In fact, one can easily check that this fraction is less than $-1/9$ on the entire interval $\be \in (0,1]$; see Figure~\ref{fig:strong_cond_plot}. Note that $d=2sk+r \geq 9$ because $k \geq 2$ and $r \geq 1$. Thus $-\frac{1}{9} \leq -\frac{1}{d}$, and we are done.

\begin{figure}
\begin{center}
\includegraphics[width=10cm]{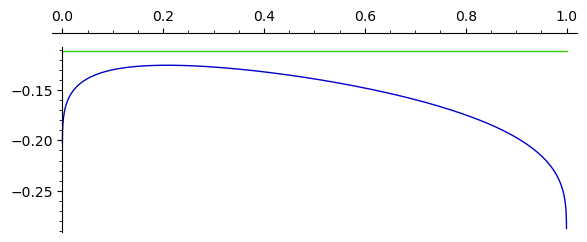}
\end{center}
\caption{The function $\ds F\left( \be, \frac{1+2\be}{2+\be} \right) \bigg/ H(\be)$ compared to $\ds -\frac19$, proving that the strong condition holds whenever $\ds k<\frac{d}{4}$}
\label{fig:strong_cond_plot}
\end{figure}

Now we turn to the case $d \geq k/4$, when we have $s=1$, $d=2k+r$, and hence 
\[ \frac{d-2k+r}{d} = \frac{2r}{2k+r} = \frac{2\be}{1+\be} .\]
So in this case we need to consider the fraction  
\[ \Ga(\be) \defeq 
F\left( \be, \frac{2\be}{1+\be} \right) \bigg/ H(\be) . \] 
This turns out to be strictly monotone decreasing in $\be$, with $\lim_{\be \to 0+} \Ga(\be) = 0$. In fact, we could easily deduce from this property alone that condition \eqref{eq:strong_cond} holds if $k \leq \big( \frac12 - o_d(1) \big) d$. In order to get a more precise threshold, we will use our estimates for $F(x,t)$.

Beforehand, note that 
\begin{align*}
\text{if } k \leq \frac{d}{3} \text{, then } 
\be &\geq 0.5 \text{, hence } 
\Ga(\be) \leq \Ga(0.5) \approx -0.040852 < -\frac{1}{25}
\text{; and} \\
\text{if } k \leq \frac{5d}{11}  \text{, then } 
\be & \geq 0.1 \text{, hence } 
\Ga(\be) \leq \Ga(0.1) \approx -0.005378 < -\frac{1}{186}
\text{.} 
\end{align*}

It follows that if $k \leq d/3$, then the strong condition holds provided that $d \geq 25$. We are also done if $d \leq 186$ and $\be \geq 0.1$. As we pointed out, we may assume that $d>500$ so we will also assume that $\be < 0.1$.

Using Proposition~\ref{prop:F_main_term} with $x=\be < 0.1$ and $t=2\be/(1+\be)$ we get the following estimate:
\begin{equation*}
F\left( \be, \frac{2\be}{1+\be} \right) \leq 
\frac{\be^2}{1+\be} \left( 1- \frac{1+\be}{2} + 
\underbrace{\log(1+\be)}_{\leq \be} 
- \log 2 \right) 
\leq \frac{\be^2}{1+\be} \left( 
\frac12 - \log 2 + \frac{\be}{2} \right) .
\end{equation*}
Therefore, using $h(1-\be) \leq \be$, we get that 
\begin{equation} \label{eq:dFplusH}
d \cdot F\left( \be, \frac{2\be}{1+\be} \right) + 
\underbrace{h(\be) + h(1-\be)}_{=H(\be)} 
\leq \frac{d\be^2}{1+\be} \left( \frac12 - \log 2 + \frac{\be}{2}  \right) + \be \log \frac{1}{\be} + \be .
\end{equation}
We need to see that this is negative. First we consider the case $r \geq 7 \log d$. Since $\be <0.1$, we have 
\[ \frac12 - \log 2 + \frac{\be}{2} 
< 0.5- \log 2 + 0.05 <-1/7 .\]
Also note that 
\[ \frac{\be}{1+\be} = \frac{2k \be}{2k(1+\be)} = 
\frac{r}{2k+r} = \frac{r}{d} .\]
It follows from \eqref{eq:dFplusH} that 
\begin{multline*}
d \cdot F\left( \be, \frac{2\be}{1+\be} \right) + H(\be) 
< \be\bigg( -r/7 + \log \frac{1}{\be} + 1 \bigg) 
= \be\bigg( -r/7 + \log \frac{d-r}{r} + 1 \bigg) \\
< \be\bigg( -r/7 + \log \frac{d}{r} + 1 \bigg) 
= \be \big( -r/7 + \log d \big) + \be\big( 1- \log r \big) ,
\end{multline*}
where both terms are negative provided that $r \geq 7 \log d$.

Finally, we show our strongest threshold: let 
\[ C_0 \defeq \frac{1}{\log 2 - 1/2} \approx 5.1774 
\text{ so that } 
C_0\left(\frac12 - \log 2\right)+1 = 0 .\]
Now choose $C$ such that $r=C \log d$ and assume that $C_0 \leq C \leq 7$. Then 
\begin{align*}
\frac{d\be}{1+\be} &= r = C \log d; \\ 
\log \frac{1}{\be} + 1 &< \log \frac{d}{r} + 1 
= \log \frac{d}{C \log d} + 1 
= \log d - \log \log d 
+ \underbrace{1 - \log C}_{\leq 1-\log C_0} .
\end{align*}
Substituting these into \eqref{eq:dFplusH} we get 
\begin{multline*}
d \cdot F\left( \be, \frac{2\be}{1+\be} \right) + H(\be) 
< \be C \log d \left( \frac12 - \log 2 + \frac{\be}{2}  \right) + \be \left( \log d - \log \log d + 1 - \log C_0 \right) \\
= \be\log d \bigg( C \big( 1/2 - \log 2 \big) - 1 \bigg) 
+ \frac12 \be^2 C \log d 
+ \be \bigg(- \log \log d + 1 - \log C_0 \bigg) .
\end{multline*}
Here the first term is non-positive since $C \geq C_0$, while the remaining part is negative because, using $C \leq 7$, we have
\[ \frac12 \be C \log d 
< \frac{7}{2} \frac{7 (\log d)^2}{d-7 \log d} 
< \log\log d - 1 + \log C_0 ,\]
where the last inequality can be checked to hold for $d \geq 405$.

\subsection{The weak condition}
\label{sec:weak}

Here we prove the last missing ingredient, Lemma~\ref{lem:profile}.

Suppose that we have $x_-,x_+$ such that \eqref{eq:x_minus_plus} holds. It follows from Proposition~\ref{prop:F_properties}(iii) that \eqref{eq:x_minus_plus} is a monotone property in the sense that if it holds for some $x_-$ and $x_+$, then it also holds for any smaller $x_-$ and any larger $x_+$. For specific values of $d,k$ it is easy to find suitable $x_-,x_+$ with a computer. 

Furthermore, let $K$ denote the set of $(x_1,x_2)$ satisfying \eqref{eq:region}. Note that $K$ is compact. The ``$t$-value'' corresponding to a point $(x_1,x_2) \in K$ is 
\[ t(x_1,x_2) \defeq 
\frac{2r}{d} + \frac{2kx_1}{d(x_1+x_2)} .\]
Using this notation \eqref{eq:eta} turns into 
\[ \eta(x_1,x_2) \defeq 
d \cdot F\bigg( x_1+x_2, t(x_1,x_2) \bigg)
+g(\al_1,x_1)+g(\al_2,x_2) .\]
The statement of the lemma is that $\eta$ is negative over $K$. Note that the definition of $K$ involves the condition $x_1/x_2 \leq \al_1 / \al_2$, which ensures that $t(x_1,x_2) \leq 1$.

For a fixed $x \in [x_-,x_+]$ we define 
\[ K_x \defeq \big\{ (x_1,x_2)\in K \, : \, x_1+x_2=x \big\} .\]

\medskip

\noindent\textbf{First case:} $x \leq \al_2$. 
We parameterize the points $(x_1,x_2) \in K_x$ using a variable $y \geq 0$:
\[ x_1=y 
\quad \text{and} \quad 
x_2=x-y .
\]
The $t$-value corresponding to the point $(y,x-y)$ is
\[ t_y \defeq t(y,x-y) = 
\underbrace{\frac{2r}{d}}_{t_0} + \frac{2k y}{dx} .\]
Then, for any fixed $x$, we introduce the following one-variable variant of $\eta$:
\begin{equation} \label{eq:eta_x}
\eta_x(y) \defeq \eta(y,x-y) = 
d \cdot F(x,t_y )+g(\al_1,y)+g(\al_2,x-y) .
\end{equation}
Using the formula for $\partial_t F (x,t)$, see Proposition~\ref{prop:F_properties}(ii), we get that 
\[ \partial_y \big( d \cdot F(x,t_y) \big) = 
\underbrace{d \frac{2k}{dx} \frac{x}{2}}_{=k} 
\log \left( 1- \frac{t_y-x}{t_y \big( 1-(2-t_y)x \big)} \right). \]
It is easy to see that the term 
\[ \frac{t_y-x}{t_y \big( 1-(2-t_y)x \big)} 
= \frac{\frac1x-\frac1{t_y}}{\frac1x-(2-t_y)} \]
is monotone increasing in $t_y$ provided that $x \leq t_y \leq 1$, which holds now as $t_y \geq t_0 \geq \al_2 \geq x$. Therefore, we have the largest derivative at $y=0$, that is 
\[ \partial_y \big( d \cdot F(x,t_y) \big) 
\leq k \log c_0 
\text{, where} \quad 
c_0 \defeq 1-\frac{t_0-x}{t_0 \big( 1-(2-t_0)x \big)} < 1 .\]
Consequently, 
\[ \eta_x(y) \leq d \cdot F(x,t_0) + 
(k \log c_0 ) y + g(\al_1,y) + g(\al_2,x-y) .\]
To find out the maximum of the right-hand side, we differentiate it w.r.t.~$y$: 
\[ k \log c_0 
+ \log \frac{(\al_1-y)(x-y)}{y(\al_2-x+y)} ,\]
where we used that 
\[ \partial_x g(\al,x)= -\log(x) + \log(\al-x) .\]
So the maximum is attained at the unique positive solution $\yt$ of the quadratic equation 
\begin{equation} \label{eq:quadratic}
y(\al_2-x+y) = c_0^k (\al_1-y)(x-y) .
\end{equation}
Specifically, we have $Ay^2+By+C=0$ with 
\[ A=1-c_0^k > 0; \, 
B=(\al_2-x)+c_0^k (\al_1+x)>0; \, 
C=  - c_0^k \al_1 x < 0 .\]
So one may use the quadratic formula to express $\yt$ explicitly. We omit this but we conclude that  
\begin{equation} \label{eq:bound_with_yt}
\max_{(x_1,x_2) \in K_x} \eta(x_1,x_2) 
= \max_{y} \eta_x(y) \leq d \cdot F(x,t_0) + 
(k \log c_0 ) \yt + g(\al_1,\yt) + g(\al_2,x-\yt) ,
\end{equation}
where $c_0$ and $\yt$ can be explicitly expressed in terms of $x$. Hence, the resulting upper bound is a concrete function of $x$. For any given pair $d,k$, using a computer it is easy to (first find a suitable $x_-$ and then) check that the afore-mentioned function is negative on $[x_-,\al_2]$. In the range $d \leq 500$ we verified this for all $d,k$ for which the strong condition fails, i.e., when $k_d^{\textrm{sc}} < k < d/2-1$. For instance, for the pair $d=99$, $k=48$ ($r=3$) we may choose $x_-=0.002$ and we get the following plot\footnote{For $r=2$ this plot goes above zero at some point, and this is where our argument fails for $k=d/2-1$.} for our upper bound as a function of $x \in [x_-,\al_2]$:
\begin{center}
\includegraphics[width=10cm]{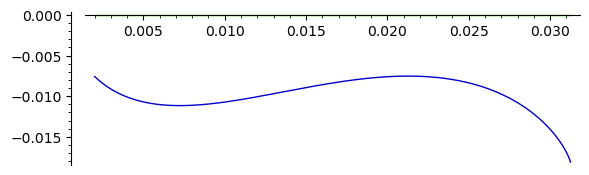}
\end{center}

We now prove that this upper bound function is negative in the range $d>500$ as well. We can still assume that the strong condition does not hold. In particular, $k>d/3$ and $r \leq 5.2 \log d < d/11$, where the second inequality is true for any $d\geq 333$. Therefore, $r/d < 1/11$, and hence  
\[ \be \defeq \al_2 = \frac{r}{2k} = \frac{r}{d-r} < 0.1 .\] 
So from this point on we assume that $\be=\al_2<0.1$.

\begin{claim*}
The choice
\[ x_-= \frac{2}{e} \frac{r}{d\sqrt{d}} .\]
satisfies \eqref{eq:x_minus_plus}.
\end{claim*}
\begin{proof}
Let $a=1/2$. Then we have 
\[ x_-= \frac{2}{e} \frac{r}{d^{1+a}} 
\quad \text{and} \quad 
\frac{x_-}{t_0} = \frac1{e d^a} .
\]
Using Proposition~\ref{prop:F_main_term} we get 
\[ F(x_-,t_0)
\leq \frac12 x_- t_0 \left( 1- \frac{1}{e d^a} - \log e - a \log d \right) 
< -\frac12 x_- t_0 a \log d. \]
Furthermore, 
\[ H(x_-) 
\leq x_-\left( \log \frac{1}{x_-} + 1 \right) 
= x_- \bigg( 1 + (1+a) \log d - \log 2 - \log r \bigg) 
< x_- (1+a) \log d.\]
Therefore, using $d t_0/2=r$, we get 
\[ F_d(x_-,t_0) = d \cdot F_d(x_-,t_0) + H(x_-) 
< -\big( ra - (1+a) \big) x_- \log d \leq 0 \]
provided that $a \geq  1/(r-1)$, which is true for our choice $a=1/2$ because $r \geq 3$. 
\end{proof}
\begin{claim*}
We have $\ds c_0 \leq \frac12$.
\end{claim*}
\begin{proof}
We need to show that 
\[ \frac{t_0-x}{t_0 \big( 1-(2-t_0)x \big)} \geq \frac12 ,\]
which is equivalent to 
\begin{equation} \label{eq:c0_cond}
\frac{t_0}{x} \geq 2- t_0(2-t_0) .
\end{equation}
Since $t_0$ does not depend on $x$, it suffices to prove \eqref{eq:c0_cond} in the worst case $x=\al_2=\beta$, when we get 
\[ \frac{2}{1+\be} 
\geq 2 - \frac{2\be}{1+\be} \frac{2}{1+\be} ,\]
which can be checked to be true for any $0 \leq \be \leq 1$.
\end{proof}

\begin{claim*}
We have $\yt \leq 2^{-k/2}$.
\end{claim*}
\begin{proof}
Recall that $\yt$ is the unique positive root of the equation \eqref{eq:quadratic}. The left-hand side is at least $y^2$, while the right-hand side is at most $c_0^k \al_1 \al_2$. It follows that 
\[ \yt \leq c_0^{k/2} \sqrt{\al_1 \al_2} \leq 2^{-k/2} .\]
\end{proof}
\begin{claim*}
If $d \geq 300$ and $\al_2=\be < 0.1$, then for every $x \leq \al_2$:
\[ \eta(0,x) = d \cdot F(x,t_0) + g(\al_2,x)
< - \frac{5r}{e d\sqrt{d}} .\]
\end{claim*}
\begin{proof}
We choose $\ga$ such that $x = \ga t_0$. Note that $\be= \frac{1+\be}{2} t_0$. Since $x\leq \be$, we have $\ga \leq (1+\be)/2 < 0.55$. 

Using $h(ab)=a \cdot h(b) + b \cdot h(a)$, we get 
\begin{multline*}
g(\be,x) = h(x)+h(\be-x)-h(\be)  
= h(\ga t_0)
+ h\left( \left(\frac{1+\be}{2}-\ga\right)t_0 \right)
- h\left( \frac{1+\be}{2} t_0 \right) \\
= t_0 \bigg( h(\ga)
+h\left(\frac{1+\be}{2}-\ga\right)
-h\left(\frac{1+\be}{2}\right) \bigg) 
\leq t_0 \bigg( h(\ga)
+h\left(0.55-\ga\right)
-h\left(0.55\right) \bigg) ,
\end{multline*}
where the last inequality is true because $h(z-\ga)-h(z)$ is monotone increasing in $z$ and $(1+\be)/2 < 0.55$.

Since $x \leq \be \leq 0.1$, we can apply Proposition~\ref{prop:F_main_term} with $x$ and $t_0=2\be/(1+\be) \geq 2x/(1+x)$. We get the following upper bound, using $t_0=2r/d$:
\begin{multline*}
d \cdot F(x,t_0) + g(\be,x) \leq 
\frac{d}{2}\ga t_0^2 \big( 1- \ga + \log \ga \big) + 
t_0 \bigg( h(\ga)+h(0.55-\ga)-h(0.55) \bigg) \\
= t_0 \bigg( r\ga \big( 1- \ga + \log \ga \big) 
+ h(\ga)+h(0.55-\ga)-h(0.55) \bigg)\\ 
\leq \ga t_0 \bigg( (r-1)\log \ga + r +1+\log(0.55) \bigg) 
= \ga t_0 \bigg( (r-1)\big(\log \ga+1) + 2+\log(0.55) \bigg).
\end{multline*} 
Setting $r=3$, the second line can be checked to be below $-0.056 t_0$ for all $\ga \in [0.05,0.55]$, and hence the same is true for any $r \geq 3$ as well. On the other hand, for $\ga \leq 0.05$, the last bound is less than 
\[ -2.5 \ga t_0 = -2.5 x \leq -2.5 x_- = - \frac{5r}{e d\sqrt{d}} ,\]  
which is precisely the stated bound. It is easy to check that the other bound $-0.056 t_0$ is even better provided that $d \geq 300$.
\end{proof}

Furthermore, according to \eqref{eq:g_bound}, we have 
\[ g(\al_1,\yt) \leq h(\yt) + \yt .\]
Since $x \leq \al_2 \leq 1/e$, we also have 
\[ g(\al_2,x-\yt) - g(\al_2,x) =
\underbrace{h(x-\yt)-h(x)}_{\leq 0} + 
\underbrace{h(\al_2-x+\yt)-h(\al_2-x)}_{\leq h(\yt)} 
\leq h(\yt) .\]
Note that in our bound \eqref{eq:bound_with_yt} the term $(k \log c_0) \yt$ is negative, and hence can be omitted. Therefore, using the previous claim, as well as $r \geq 3$ and $d/3 < k < d/2$, we get  
\[\max \eta_x 
\leq - \frac{5r}{e d\sqrt{d}} + 2h(\yt) + \yt 
\leq - \frac{15}{e d\sqrt{d}} 
+ \big(1+ k \log 2 \big) 2^{-k/2} 
\leq - \frac{15}{e d\sqrt{d}} + \frac{\log 2}{2} d2^{-d/6} ,\]
which is negative for $d \geq 68$, and we are done.

\medskip

\noindent\textbf{Second case:} $x > \al_2$. 

The proof in this case goes along similar lines but our estimates do not need to be so sharp. This time we parameterize the points $(x_1,x_2)$ in $K_x$ slightly differently: for $y \geq 0$ let 
\[ x_1=x-\al_2+y 
\quad \text{and} \quad 
x_2=\al_2-y .\]
Accordingly, there is a shift in $t_y$ as well:
\[ t_y \defeq 
\underbrace{\frac{2r}{d}+\frac{2k(x-\al_2)}{dx}}_{t_0} + \frac{2k y}{dx} .\]
In this case we define $\eta_x$ as
\[ \eta_x(y) \defeq \eta(x-\al_2+y,\al_2-y) = 
d \cdot F(x,t_y )+
\underbrace{g(\al_1,x-\al_2+y)}_{=g(\al_1,1-x-y)}+
\underbrace{g(\al_2,\al_2-y)}_{=g(\al_2,y)} .\]
It is easy to see that we still have $1 \geq t_y \geq t_0 \geq x$, and hence the following is valid in this case as well:
\[ \partial_y \big( d \cdot F(x,t_y) \big) 
\leq k \log c_0 
\text{, where} \quad 
c_0 \defeq 1-\frac{t_0-x}{t_0 \big( 1-(2-t_0)x \big)} < 1 .\]
Consequently, 
\[ \eta_x(y) \leq d \cdot F(x,t_0) + 
(k \log c_0 ) y + g(\al_1,1-x-y) + g(\al_2,y) .\]
The derivative of the right-hand side w.r.t.~$y$ is 
\[ k \log c_0 
+ \log \frac{(1-x-y)(\al_2-y)}{y(x-\al_2+y)} .\]
So the maximum is attained at the unique positive solution $\yt$ of the quadratic equation 
\[ y(x-\al_2+y) = c_0^k (1-x-y)(\al_2-y) .\]
We arrive at the following bound:
\begin{equation} \label{eq:bound_case2}
\max_{y} \eta_x(y) \leq d \cdot F(x,t_0) + 
\underbrace{(k \log c_0 )}_{<0} \yt + g(\al_1,1-x-\yt) + g(\al_2,\yt) .
\end{equation}
As in the first case, one can explicitly express $c_0$ and $\yt$ in terms of $x$. We checked with a computer that the resulting upper bound is negative for all $d,k$ in the range $d \leq 500$.

Now we turn to the range $d>500$. As we saw in the first case, we may also assume that $\al_2=\be<0.1$. 

First of all, $\yt \leq c_0^{k/2}$ still holds by the very same argument as in the first case. We claim that we still have $c_0 \leq 1/2$, too. As we have seen, this is equivalent to \eqref{eq:c0_cond}, which we checked to be true at $x=\al_2$. Taking a larger $x$, $t_0$ gets larger, and hence the right-hand side of \eqref{eq:c0_cond} gets smaller, while the left-hand side becomes larger according to the next claim.
\begin{claim*}
At $x=\al_2$ we have 
\[ \frac{t_0}{x} = \frac{2}{1+\be} = \frac{4k}{d} ,\]
while for $\ds \al_2< x < \sqrt{ \frac{r}{2d}}$ we have 
\[ \frac{t_0}{x} > \frac{2\be}{1+\be} .\]
%
%
%
\end{claim*}
\begin{proof}
Note that $t_0/x$ at $x=\al_2$ is equal to 
\[ \frac{2r}{d} \bigg/ \frac{r}{2k} = \frac{4k}{d} 
= \frac{2}{1+\be} < 2.\] 
The derivative of $t_0$ w.r.t.~$x$ is 
\[ \frac{2k \al_2}{dx^2} > 2 \text{ because }
x < \sqrt{\frac{k \al_2}{d}} = \sqrt{\frac{r}{2d}} .\]
It follows that 
\[ \frac{t_0}{x} > \frac{2}{1+\be} 
\text{ for any } \al_2< x < \sqrt{ \frac{r}{2d}} .\]
%
%
%
\end{proof}

Therefore, if $x \leq 0.2$, then we can apply Proposition~\ref{prop:F_main_term} to get  
\[ d \cdot F(x,t_0) \leq 
\frac{d}{2} x t_0 \varphi\left( \frac{x}{t_0} \right) .\]
Note that 
\[ \frac{d}{2} x t_0 = rx + k(x-\al_2) .\]
Furthermore, by the previous claim  
\[ \frac{t_0}{x} \geq \frac{2}{1+\be} 
\text{, and hence } 
\frac{x}{t_0} \leq \frac{1+\be}{2} \leq 0.55 .\]
Using the monotonicity of $\varphi$ we conclude that 
\[ d \cdot F(x,t_0) \leq 
\frac{d}{2} x t_0 \, \varphi\left( \frac{x}{t_0} \right) 
\leq \underbrace{\varphi(0.55)}_{\leq - 0.14} 
\big( rx + k(x-\al_2) \big) .\]

On the first hand, due to \eqref{eq:g_bound} and $\al_2 \leq 0.1 < 1/e$, we have
\[ g(\al_2,\yt) \leq h(\yt) 
\leq \frac{k}{2} (\log 2) 2^{-k/2}.\]
On the other hand, 
\begin{multline*}
g(\al_1,1-x-\yt) = h(x-\al_2+\yt) + h(1-x-\yt) - h(\al_1) \\
\leq h(x-\al_2)+h(\yt) + h(1-x) + \yt - h(1-\al_2)
\leq h(x-\al_2)+(x-\al_2)+h(\yt)+\yt .
\end{multline*} 
Using \eqref{eq:bound_case2} and the bounds above, we conclude that 
\[ \max \eta_x \leq - 0.14 \big( rx + k(x-\al_2) \big) 
+ h(x-\al_2)+(x-\al_2)+2h(\yt)+\yt .\]
Note that 
\[ - 0.14 k(x-\al_2) + h(x-\al_2)+(x-\al_2) \leq 0 
\text{ if } x-\al_2 \geq \exp\big( 1- 0.14 k \big) .\]
Otherwise 
\[ h(x-\al_2) + (x-\al_2) 
< 0.14 k \exp\big( 1- 0.14 k \big) .\]
So in both cases we have the following upper bound:
\[ - 0.14 k(x-\al_2) + h(x-\al_2)+(x-\al_2) 
\leq 0.14 k \exp\big( 1- 0.14 k \big) .\]
Furthermore, 
\[ - 0.14 rx 
\leq - 0.14 r \al_2 
\leq -0.14 \frac{r^2}{2k} 
\leq -0.14 \frac{3^2}{2k} 
= - 0.63 \frac{1}{k} .\]
So it suffices to prove that 
\[ -0.63 \frac{1}{k} + \big(1+ k \log 2 \big) 2^{-k/2} + 0.14 k \exp\big( 1- 0.14 k \big) < 0 ,\]
which holds for $k \geq 54$. (Recall that we work under the assumptions $k>d/3$ and $d>500$.)

The only further assumptions we used in the second case are that $x < \sqrt{r/(2d)}$ and $x<0.2$. Consequently, it remains to show that $x_+=\min\big(0.2, \sqrt{r/(2d)}\big)$ satisfies \eqref{eq:x_minus_plus}. Note that  
\[ \sqrt{\frac{r}{2d}} \leq \sqrt{\frac{5.2 \log d}{2d}} 
< 0.2 \text{ for } d \geq 388 .\]
So it suffices to check that $x_+= \eps \defeq \sqrt{r/(2d)}<0.2$ is a good choice, i.e., we need that 
\[ F_d(1-x_+,\underbrace{2k/d}_{=1-r/d}) 
= d \cdot F(1-\eps,1-2\eps^2) + H(\eps) < 0 .\]
Since $d = r/(2 \eps^2) \geq 3/(2 \eps^2)$, this follows from 
\[ \frac{3}{2 \eps^2} \cdot F(1-\eps,1-2\eps^2) + H(\eps) < 0 ,\]
which can be checked to hold for $0<\eps<0.3$.

\bibliographystyle{plain}
\bibliography{refs}

\end{document}